\documentclass[11pt,a4paper]{article}

\usepackage{cmap}
\usepackage[T1]{fontenc}

\usepackage{amsmath,amsthm,amssymb,latexsym,graphicx,mathrsfs}
\usepackage{secdot}
\sectiondot{subsection}
\usepackage{a4wide}
\usepackage{enumitem}

\usepackage{xcolor}
\usepackage{url}
\usepackage{hyperref}
\definecolor{ForestGreen}{rgb}{0.1,0.6,0.05}
\definecolor{EgyptBlue}{rgb}{0.063,0.1,0.6}
\definecolor{RipeOlive}{HTML}{556B2F}
\hypersetup{
	colorlinks=true,
	linkcolor=EgyptBlue,         
	citecolor=ForestGreen,
	urlcolor=RipeOlive
}

\usepackage[hyperpageref]{backref}
\usepackage{epstopdf}
\newtheorem{theorem}{Theorem}
\newtheorem{proposition}[theorem]{Proposition}
\newtheorem{lemma}[theorem]{Lemma}

\theoremstyle{definition}

\newtheorem{remark}[theorem]{Remark}

\numberwithin{equation}{section}
\numberwithin{theorem}{section}

\numberwithin{equation}{section}
\numberwithin{theorem}{section}

\newenvironment{proof*}[1]{\begin{trivlist}\item[\hskip%
		\labelsep{{\bf Proof of \/{\rm\bf #1.}}\quad}]\rm}%
	{\hfill\qed\rm\end{trivlist}}

\newcommand{\W}{W_0^{1,p}}

\newcommand{\intO}{\int_\Omega}

\newcommand{\E}{E_{\alpha,\beta}}
 
\newcommand{\J}{J_{\alpha,\beta}}
\newcommand{\N}{\mathcal{N}_{\alpha,\beta}}
\newcommand{\Ep}{\widetilde{E}_{\alpha,\beta}}

\DeclareRobustCommand\alphaM{\alpha}
\DeclareRobustCommand\betaM{\beta}

\title{
	\vspace*{-1cm}
	Multiplicity of positive solutions for $(p,q)$-Laplace equations with two parameters} 
\author{ 
	\normalsize Vladimir Bobkov\\ 
	{\small  Department of Mathematics and NTIS, Faculty of Applied Sciences, University of West Bohemia}\\ 
	{\small Univerzitn\'i 8, 301 00 Plze\v{n}, Czech Republic}\\
	{\small  Institute of Mathematics, Ufa Federal Research Centre, RAS}\\ 
	{\small Chernyshevsky str.\ 112, 450008 Ufa, Russia}\\
	{\small e-mail: bobkov@kma.zcu.cz}\\[0.5em] 	
	\normalsize Mieko Tanaka\\
	{\small Department of  Mathematics, 
		Tokyo University of Science}\\
	{\small Kagurazaka 1-3, Shinjyuku-ku, Tokyo 162-8601, Japan}\\
	{\small e-mail: miekotanaka@rs.tus.ac.jp} 
}

\date{}

\begin{document}
\maketitle 
	
\begin{abstract} 
	
	We study the zero Dirichlet problem for the equation $-\Delta_p u -\Delta_q u = \alpha |u|^{p-2}u+\beta |u|^{q-2}u$ in a bounded domain $\Omega \subset \mathbb{R}^N$, with $1<q<p$.
	We investigate the relation between two critical curves on the $(\alpha,\beta)$-plane corresponding to the threshold of existence of special classes of positive solutions.
	In particular, in certain neighbourhoods of the point 
	$(\alpha,\beta) = \left(\|\nabla \varphi_p\|_p^p/\|\varphi_p\|_p^p,\,\|\nabla \varphi_p\|_q^q/\|\varphi_p\|_q^q\right)$, 
	where $\varphi_p$ is the first eigenfunction of the $p$-Laplacian, 
	we show the existence of two and, which is rather unexpected, three distinct positive solutions, depending on a relation between the exponents $p$ and $q$. 
			
	
\end{abstract} 

\section{Introduction and main results} 
We consider the boundary value problem
\begin{equation}\label{eq:D}
\tag{$D_{\alphaM,\betaM}$}
\left\{
\begin{aligned}
-\Delta_p u -\Delta_q u &= \alpha |u|^{p-2}u+\beta |u|^{q-2}u 
&&\text{in}\ \Omega, \\
u&=0 &&\text{on}\ \partial \Omega,
\end{aligned}
\right.
\end{equation}
where the operator $\Delta_r$, formally defined as $\Delta_r u = \text{div}\left(|\nabla u|^{r-2} \nabla u \right)$ for $r = p,q > 1$, is the $r$-Laplacian, $\alpha, \beta \in \mathbb{R}$ are parameters, and $\Omega \subset \mathbb{R}^N$ is a bounded domain, $N \geq 1$. In the case $N \geq 2$, we require the boundary $\partial \Omega$ of $\Omega$ to be $C^2$-smooth.
Throughout the text, we always assume $q < p$, which involves no loss of generality.

The differential operator in the problem \eqref{eq:D} is usually called the $(p,q)$-Laplacian, and thereby \eqref{eq:D} can be formally understood as the corresponding eigenvalue problem. 
Although the presence of two spectral parameters ($\alpha$ and $\beta$) is not typical in nonlinear spectral theories (cf.\ \cite{APV,FNSS}), such choice appears to be more convenient for our particular problem since it provides a separate control of the influence of the $(p-1)$- and $(q-1)$-homogeneous parts. 
Considered independently, these parts correspond to the eigenvalue problems for the $p$- and $q$-Laplacians, and it is thus natural to anticipate a strong dependence of the structure of the solution set of \eqref{eq:D} on the spectrum of both $p$- and $q$-Laplacians. 
Indeed, the problem \eqref{eq:D} has been investigated in a few works, where certain nontrivial dependences of this kind were obtained, see, e.g., \cite{cindeg,CS,KTT,MP,MT,T-2014}, the works \cite{BobkovTanaka2015,BobkovTanaka2016,BobkovTanaka2017,BobkovTakanaPicone} of the present authors, and a survey \cite{marcomasconi}. 
In the present article, we continue our investigation of the problem \eqref{eq:D} by establishing several nontrivial multiplicity results, mainly in special neighbourhoods of the point
\begin{equation*}\label{eq:ab}
(\alpha,\beta) = \left(\frac{\|\nabla \varphi_p\|_p^p}{\|\varphi_p\|_p^p},\frac{\|\nabla \varphi_p\|_q^q}{\|\varphi_p\|_q^q}\right),
\end{equation*}
where $\varphi_p$ is the first eigenfunction of the $p$-Laplacian.
In particular, we discover the formation of an $S$-shaped bifurcation diagram when $p>2q$, see Figure \ref{fig:1}.

Prior to the rigorous description of main results, let us mention that various problems with the $(p,q)$-Laplacian, whose motivation arises from both mathematical and physical premises, are actively studied in the contemporary literature. 
Among physical origins of the $(p,q)$-Laplacian, one can think of it as a formal two-term Taylor approximation of more complex differential operators, see, e.g.,
\cite{zakharov} for the Zakharov equation describing in a simplified way long-wave oscillations of a plasma, \cite{benci} for a higher-dimensional generalization of the sine-Gordon equation which possesses soliton-type solutions, and \cite{BCF} for an approximation of the electrostatic Born-Infeld equation with a superposition of point charges. 
Let us also point out a model in the theory of crystal growth containing the one-dimensional $(1,2)$-Laplacian which was studied in \cite{mucha}.
Among mathematical origins, the $(p,q)$-Laplacian occurs, e.g., in the procedure of elliptic regularization which consists in the inclusion of the regularizing term $\varepsilon^2 \Delta$, $\varepsilon \in \mathbb{R}$, in a nonlinear equation, with a view to obtain better properties of the augmented equation, see, for instance, \cite{agudelo,marcel}. 
An investigation of variational functionals with nonstandard $(p,q)$-growth conditions, mainly in connection with 
the Lavrentiev gap phenomenon, has been performed, e.g., in \cite{mignione,zhikov}.
Finally, we refer the interested reader to a nonexhaustive list of works \cite{CMP,CI,cindeg,MP,PVV,PSqas} for a development of the existence theory for various problems with the $(p,q)$-Laplacian.

\subsection{Several notations}
Hereinafter, we denote the Sobolev space $W_0^{1,r}(\Omega)$ shortly by $W_0^{1,r}$, where $r>1$. The standard norm of the Lebesgue space $L^r(\Omega)$ will be denoted by $\|\cdot\|_r$.
A function $u\in\W$ is called a (weak) solution of \eqref{eq:D} if the following equality is satisfied for any test function $\varphi\in \W$:
\begin{equation}\label{eq:D:weak}
\intO |\nabla u|^{p-2}\nabla u\nabla\varphi \,dx 
+ \intO |\nabla u|^{q-2}\nabla u\nabla\varphi \,dx 
= \alpha \intO |u|^{p-2}u \varphi \, dx + \beta \intO |u|^{q-2}u \varphi\,dx.
\end{equation}
The energy functional $\E : \W \to \mathbb{R}$ associated with \eqref{eq:D} is given by
\begin{equation*}\label{def:E} 
\E(u) = \frac{1}{p}\, H_\alpha (u)+\frac{1}{q}\,G_\beta(u), 
\end{equation*}
where
\begin{equation*}\label{def:HG}
H_\alpha (u) :=\|\nabla u\|_p^p -\alpha\|u\|_p^p 
\quad {\rm and}\quad 
G_\beta(u) :=\|\nabla u\|_q^q -\beta\|u\|_q^q.
\end{equation*}
Since $p>q>1$, we have $\E \in C^1(\W,\mathbb{R})$, and hence weak solutions of \eqref{eq:D} are in one-to-one correspondence with critical points of $\E$.

\begin{remark}\label{rem:positive}
	Using the Moser iteration process (see, e.g.,  
	\cite[Appendix A]{MMT}), one can show that any solution $u$ of \eqref{eq:D} belongs to $L^\infty(\Omega)$. 
	Then, the regularity up to the boundary given by \cite[Theorem~1]{Lieberman} and \cite[p.~320]{L} ensures that $u\in C^{1,\gamma}_0(\overline{\Omega})$ for some $\gamma\in(0,1)$. 
	Moreover, if $u$ is a nonzero nonnegative solution, then the strong maximum principle and the boundary point lemma (see, e.g., \cite[Theorems 5.4.1 and 5.5.1]{puser}) guarantee that $u$ is positive and belongs to
	$$
	{\rm int}\,C_0^1(\overline{\Omega})_+ := 
	\left\{
	u \in C^1_0(\overline{\Omega}):~
	u(x)>0 \text{ for all } x \in \Omega,~
	\frac{\partial u}{\partial\nu}(x) < 0 \text{ for all } x \in \partial\Omega 
	\right\},
	$$
	the interior of the positive cone of $C^1_0(\overline{\Omega})$.
	Here $\nu$ is the exterior unit normal vector to $\partial \Omega$. 
\end{remark}

Finally, we denote by $\lambda_1(r)$ the first eigenvalue of the $r$-Laplacian, i.e.,
$$
\lambda_1(r) 
= 
\inf\left\{\frac{\|\nabla u\|_r^r}{\|u\|_r^r}:~ u \in W_0^{1,r} \setminus \{0\}
\right\},
$$
and by $\varphi_r$ the corresponding first eigenfunction. Notice that $\varphi_r$ has a constant sign in $\Omega$, and we will assume, without loss of generality, that $\varphi_r > 0$ in $\Omega$ and $\|\nabla \varphi_r\|_r=1$. 
Moreover, for such $\varphi_r$ we have $\varphi_r \in {\rm int}\,C_0^1(\overline{\Omega})_+$.
Furthermore, since $p>q$, $\varphi_p$ cannot simultaneously be an eigenfunction of the $q$-Laplacian, see \cite[Proposition 13]{BobkovTanaka2017}.

\subsection{Overview of known results}\label{sec:knownfacts}

Let us divide the $(\alpha,\beta)$-plane into four open quadrants by the lines 
$\{\lambda_1(p)\} \times \mathbb{R}$ and $\mathbb{R} \times \{\lambda_1(q)\}$ (see Figures \ref{fig:1}, \ref{fig:2}). 
We recall several known facts about the existence, nonexistence, and multiplicity of positive solutions of \eqref{eq:D} in these quadrants, as well as on their boundaries $\{\lambda_1(p)\} \times \mathbb{R}$ and $\mathbb{R} \times \{\lambda_1(q)\}$.

\begin{proposition}[\protect{\cite[Proposition 1]{BobkovTanaka2015} and \cite[Proposition 13]{BobkovTanaka2017}}]\label{prop:nonexistence}
	Let $\alpha \leq \lambda_1(p)$ and $\beta \leq \lambda_1(q)$. 
	Then \eqref{eq:D} has no nonzero solution.
\end{proposition}

\begin{proposition}[\protect{\cite[Propositions 2 and 6]{BobkovTanaka2015} and \cite[Remark 1]{BobkovTanaka2017}; see also \cite[Lemma 2.2]{MP} for a related result}]\label{prop:uniq}
	Let $\alpha < \lambda_1(p)$ and $\beta > \lambda_1(q)$. 
	Then \eqref{eq:D} has at least one positive solution. 
	Moreover, any nonzero solution $u$ of \eqref{eq:D} satisfies $\E(u) < 0$. 
	Furthermore, if $\alpha \leq 0$, then the positive solution is unique.
\end{proposition}

\begin{proposition}[\protect{\cite[Propositions 2 and 6]{BobkovTanaka2015}}]
	Let $\alpha > \lambda_1(p)$ and $\beta < \lambda_1(q)$. 
	Then \eqref{eq:D} has at least one positive solution. Moreover, any nonzero solution $u$ of \eqref{eq:D} satisfies $\E(u) > 0$. 
\end{proposition}

In order to discuss the existence in the remaining quadrant $(\lambda_1(p),+\infty) \times (\lambda_1(q),+\infty)$, we introduce the threshold curve
$$
\beta_{ps}(\alpha) 
:= 
\sup
\left\{\beta \in \mathbb{R}:~ \eqref{eq:D} ~\text{has at least one positive solution}
\right\}
$$
for $\alpha \geq \lambda_1(p)$.
Define also the values
\begin{equation*}\label{def:values} 
\alpha_* = \frac{\|\nabla \varphi_q\|_p^p}{\|\varphi_q\|_p^p} 
\quad \text{and} \quad 
\beta_* = \frac{\|\nabla \varphi_p\|_q^q}{\|\varphi_p\|_q^q}. 
\end{equation*}
It was proved in \cite[Proposition 3]{BobkovTanaka2015} (see also \cite[Remark 4]{BobkovTanaka2017}) that $\beta_{ps}(\alpha) < +\infty$ for any $\alpha > \lambda_1(p)$, $\beta_{ps}(\lambda_1(p)) \geq \beta_*$, $\beta_{ps}(\cdot)$ is continuous and nonincreasing on $(\lambda_1(p),+\infty)$, and $\beta_{ps}(\alpha) = \lambda_1(q)$ for all $\alpha \geq \alpha_*$.
Moreover, $\alpha_* > \lambda_1(p)$ and $\beta_* > \lambda_1(q)$, see \cite[Lemma 2.1]{BobkovTanaka2017}.

\begin{theorem}[\protect{\cite[Theorem 2.2 and Proposition 4]{BobkovTanaka2015}}]\label{prop:betaps}
	Let $\alpha \in (\lambda_1(p), \alpha_*)$ and $\beta \in (-\infty,\beta_{ps}(\alpha)]$. 
	Then \eqref{eq:D} has at least one positive solution.
\end{theorem}

However, the properties of $\beta_{ps}(\alpha)$ for $\alpha \in [\lambda_1(p), \alpha_*)$ are far from being completely understood. 
In particular, the asymptotic behaviour of $\beta_{ps}(\alpha)$ as $\alpha$ approaches $\lambda_1(p)$ was substantially unclear until the recent work \cite{BobkovTakanaPicone}, where the following results have been established by obtaining a nontrivial generalization of the classical Picone inequality \cite{Alleg} and using the generalized Picone inequalities \cite[Proposition 2.9]{BF}, \cite[Lemma 1]{ilyas}, and a radial symmetry result of \cite{BrockTakac}.

\begin{theorem}[\protect{\cite[Theorem 3.3]{BobkovTakanaPicone}}]\label{thm:summary} 
	We have $\beta_* \leq \beta_{ps}(\lambda_1(p))<+\infty$. 
	Moreover, {\renewcommand{\alphaM}{\lambda_1(p)}\eqref{eq:D}} has at least one positive solution if $\lambda_1(q) < \beta < \beta_{ps}(\lambda_1(p))$.
	Furthermore, if $\beta_{ps}(\lambda_1(p)) > \beta_*$, then {\renewcommand{\alphaM}{\lambda_1(p)}\eqref{eq:D}} has at least one positive solution if and only if $\lambda_1(q) < \beta \leq \beta_{ps}(\lambda_1(p))$. 
\end{theorem} 

\begin{theorem}[\protect{\cite[Theorem 3.2]{BobkovTakanaPicone}}]\label{thm:nonexistence0}
	Assume that one of the following assumptions is satisfied:
	\begin{enumerate}[label={\rm(\roman*)}]
		\item\label{thm:nonexistence:1x} 
		$p \in I(q)$, where
		\begin{equation*}\label{def:I(q)}  
		I(q) := \{p> 1:~ (q-1) s^p + q s^{p-1} - (p-q) s + (q-p+1) \ge 0 ~\text{for all}~ s\geq 0\};
		\end{equation*}
		\item\label{thm:nonexistence:2x} $p \leq q+1$ and $\Omega$ is an $N$-ball.
	\end{enumerate}	
	Then {\renewcommand{\alphaM}{\lambda_1(p)}\eqref{eq:D}} has no positive solution for $\beta > \beta_*$, that is, $\beta_{ps}(\lambda_1(p))=\beta_*$.
	Moreover, if $p<q+1$ and $\Omega$ is an $N$-ball, then {\renewcommand{\alphaM}{\lambda_1(p)}\eqref{eq:D}} has no positive solution also for $\beta = \beta_*$. 
\end{theorem}

\begin{remark}
	We recall that $q<p$ in Theorem \ref{thm:nonexistence0} by default.
	The set $I(q)$ is characterized in \cite[Lemma 1.6]{BobkovTakanaPicone}.
	In particular, it is known that for each $q>1$ there exists $\widetilde{p} \in (\max\{2,q\},q+1)$ such that
	$[2,\widetilde{p}] \subset I(q)$ and $(\widetilde{p},+\infty) \cap I(q) = \emptyset$.
\end{remark}

Theorem \ref{thm:nonexistence0} generates a natural question on whether $\beta_{ps}(\lambda_1(p)) > \beta_*$ if either the assumption \ref{thm:nonexistence:1x} or \ref{thm:nonexistence:2x} of this theorem is violated. 
Nontriviality of this question is supported by the fact that the behaviour of the energy functional $\E$ at the point $(\lambda_1(p), \beta_*)$ crucially depends on the relation between $p$ and $q$. 
\begin{theorem}[\protect{\cite[Theorem 2.6 (ii) and Remark 5]{BobkovTanaka2017}}]\label{thm:GM}
	We have the following assertions:
	\begin{enumerate}[label={\rm(\roman*)}]
		\item\label{thm:GM:3} If $p < 2q$, then $\inf_{\W} E_{\lambda_1(p), \beta_*} = -\infty$.
		\item\label{thm:GM:2} If $p = 2q$, and $\partial\Omega$ is connected when $N\ge 2$, then $\inf_{\W} E_{\lambda_1(p), \beta_*} \in (-\infty, 0)$.
		\item\label{thm:GM:1} If $p>2q$, and $\partial\Omega$ is connected when $N\ge 2$, then $\inf_{\W} E_{\lambda_1(p), \beta_*} \in (-\infty, 0)$ and the infimum is attained by a positive solution of {\renewcommand{\betaM}{\beta_*}\renewcommand{\alphaM}{\lambda_1(p)}\eqref{eq:D}}.
	\end{enumerate}
\end{theorem}

Let us remark that the connectedness of $\partial \Omega$ is required in the proof of Theorem \ref{thm:GM} \ref{thm:GM:2}, \ref{thm:GM:1} due to the usage of the improved Poincar\'e inequality obtained in \cite{takac}. 
It is conjectured in \cite[Section 3.1]{takac}, however, that this assumption on $\partial \Omega$ can be omitted for sufficiently regular domains. 
To the best of our knowledge, this conjecture is still open.

Theorem \ref{thm:GM} suggests that not only the behaviour of $E_{\lambda_1(p), \beta_*}$ but also the structure of the solution set of \eqref{eq:D} in a neighbourhood of the point $(\lambda_1(p), \beta_*)$ is different in the cases $p<2q$, $p=2q$, and $p>2q$, which possibly affects the relation between $\beta_{ps}(\lambda_1(p))$ and $\beta_*$. Indeed, this turns out to be true, and the precise results will be formulated in Section \ref{sec:mainresults} below.

\medskip
A finer existence result in the quadrant $(\lambda_1(p),+\infty) \times (\lambda_1(q),+\infty)$ can be obtained if we introduce the following family of critical values for $\alpha \geq \lambda_1(p)$:
\begin{equation}\label{def:alpha_*}
	\beta_*(\alpha) := \inf\left\{ \frac{\|\nabla u\|_q^q}{\|u\|_q^q}:~
	u\in\W\setminus \{0\} \text{ and } H_\alpha(u)\le 0\right\},
\end{equation}
or, equivalently,
\begin{equation*}
\beta_*(\alpha) 
= 
\inf
\left\{ 
\frac{\|\nabla u\|_q^q}{\|u\|_q^q}:~
u\in\W\setminus \{0\} \text{ and } 
\frac{\|\nabla u\|_p^p}{\|u\|_p^p} \leq \alpha
\right\}.
\end{equation*}
It is known that $\beta_*(\cdot)$ is continuous and nonincreasing on $[\lambda_1(p),+\infty)$, $\beta_*(\lambda_1(p)) = \beta_* > \beta_*(\alpha) \geq \lambda_1(q)$ for $\alpha > \lambda_1(p)$, 
and $\beta_*(\alpha) > \lambda_1(q)$ if and only if $\alpha<\alpha_*$, see \cite[Proposition 7]{BobkovTanaka2017}.

\begin{theorem}[\protect{\cite[Theorem 2.7]{BobkovTanaka2017}}]\label{thm:multi}
	Let $\lambda_1(p)<\alpha<\alpha_*$ and $\lambda_1(q)<\beta\le \beta_*(\alpha)$. 
	Then \eqref{eq:D} has at least two positive solutions $u_1$ and $u_2$ such that $\E(u_1)<0$, $\E(u_2)>0$ if $\beta<\beta_*(\alpha)$, and $\E(u_2)=0$ if $\beta=\beta_*(\alpha)$. 
	Moreover, $u_1$ is the global least energy solution, and if $\beta<\beta_*(\alpha)$, then 
	$u_2$ has the least energy among all solutions $w$ of \eqref{eq:D} with $\E(w) > 0$.
\end{theorem}

In particular, Theorem \ref{thm:multi} yields
\begin{equation}\label{eq:betabeta}
\beta_*(\alpha)\le \beta_{ps}(\alpha)
\quad 
\text{ for all~ } 
\alpha\ge \lambda_1(p).
\end{equation}
Moreover, $\beta_*(\alpha) = \beta_{ps}(\alpha) = \lambda_1(q)$ for all $\alpha \geq \alpha_*$. 
The most essential open question here was whether the strict inequality in \eqref{eq:betabeta} holds. 
This issue is addressed in the present article, see the following subsection.

\medskip
Finally, in accordance with the results described above, the only place on the $(\alpha,\beta)$-plane where it remains to discuss the existence of positive solutions of \eqref{eq:D} is the interval $[\alpha_*,+\infty) \times \{\lambda_1(q)\}$.
It was proved in \cite[Proposition 4 (ii)]{BobkovTanaka2015} that {\renewcommand{\betaM}{\lambda_1(q)}\eqref{eq:D}} has no positive solution whenever $\alpha > \alpha_*$. 
In fact, the proof of \cite[Proposition 4 (ii)]{BobkovTanaka2015} can be slightly updated in order to show that the nonexistence persists also in the case $\alpha=\alpha_*$. 
Indeed, it follows from the proof of \cite[Proposition 4 (ii)]{BobkovTanaka2015} that if  {\renewcommand{\betaM}{\lambda_1(q)}\renewcommand{\alphaM}{\alpha_*}\eqref{eq:D}} possesses a positive solution $u$, then $u = k\varphi_q$ for some $k >0$. However, this is impossible in view of \cite[Proposition 13]{BobkovTanaka2017}.
Thus, thanks to Proposition \ref{prop:nonexistence} and Theorem \ref{prop:betaps}, {\renewcommand{\betaM}{\lambda_1(q)}\eqref{eq:D}} possesses a positive solution if and only if $\alpha \in (\lambda_1(p), \alpha_*)$.

\subsection{Statements of main results}\label{sec:mainresults}

For convenience, we introduce the following hypothesis:
\begin{enumerate}[label={(H)}]
	\item\label{H} ~$p>2q$, and if $N \geq 2$, then $\partial\Omega$ is connected. 
\end{enumerate}

\begin{figure}[ht]
	\centering
	\includegraphics[width=0.7\linewidth]{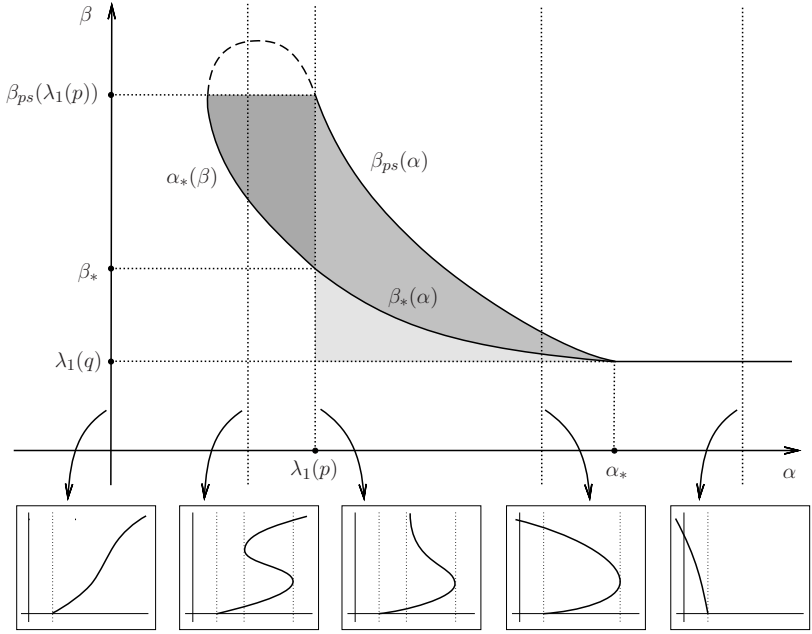}\\
	\caption{The case $p>2q$. The behaviour of $\beta_*(\alpha)$, $\beta_{ps}(\alpha)$, $\alpha_*(\beta)$, and alleged ``minimal'' bifurcation diagrams for the $L^\infty$-norms of positive solutions of \eqref{eq:D} with respect to $\beta$ for several fixed $\alpha$'s. 
	Light grey - two positive solutions, one of which is with positive energy and another one is with negative energy; grey - two positive solutions with negative energy; dark grey - three positive solutions with negative energy.}
	\label{fig:1}
\end{figure}

\begin{figure}[ht]
	\centering
	\includegraphics[width=0.7\linewidth]{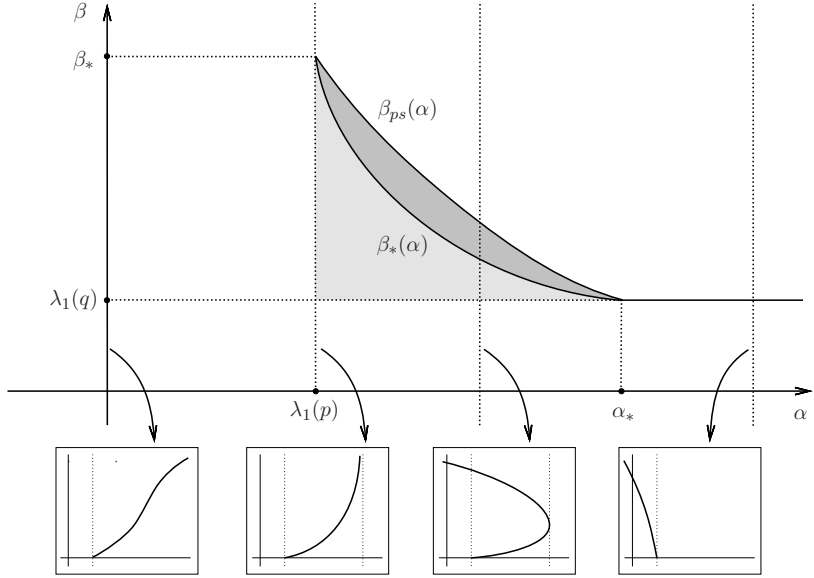}\\
	\caption{$p,q$ as in Theorem \ref{thm:nonexistence0}. 
	The behaviour of $\beta_*(\alpha)$, $\beta_{ps}(\alpha)$, and alleged ``minimal'' bifurcation diagrams for the $L^\infty$-norms of positive solutions of \eqref{eq:D} with respect to $\beta$ for several fixed $\alpha$'s. Light grey - two positive solutions, one of which is with positive energy and another one is with negative energy; grey - two positive solutions with negative energy.}
	\label{fig:2}
\end{figure}

\noindent
Our first result is devoted to the relation between $\beta_{ps}(\alpha)$ and $\beta_*(\alpha)$.
\begin{theorem}\label{thm:1}
Let $\alpha\in [\lambda_1(p), \alpha_*)$, 
and assume \textnormal{\ref{H}} if $\alpha=\lambda_1(p)$. 
Then there exists 
$\tilde{\beta}(\alpha)>\beta_*(\alpha)$ such that \eqref{eq:D} possesses a positive solution $u$ for any $\beta \in (\beta_*(\alpha), \tilde{\beta}(\alpha)]$. Moreover, $u$ is a local minimum point of $\E$ and $\E(u) < 0$. 
\end{theorem}

The idea of the proof of Theorem \ref{thm:1} is based on the recent work \cite{IK}, where the authors obtained a local continuation of the branch of least energy solutions of an elliptic problem with indefinite nonlinearity using an original variational argument of a constrained minimization type.

Theorem \ref{thm:1} implies, in particular, that 
\begin{equation*}\label{eq:b<b0}
\beta_*(\alpha) < \beta_{ps}(\alpha)
\quad 
\text{ for all~ } 
\alpha \in (\lambda_1(p), \alpha_*),
\end{equation*}
and that $\beta_* < \beta_{ps}(\lambda_1(p))
$ provided the additional assumption \ref{H} is satisfied. 
On the other hand, we know from Theorem \ref{thm:nonexistence0} that $\beta_* = \beta_{ps}(\lambda_1(p))$ if either $p \in I(q)$ or $p \leq q+1$ and $\Omega$ is an $N$-ball.
Nevertheless, it remains unknown whether  $\beta_*=\beta_{ps}(\lambda_1(p))$ for all $p \leq 2q$ regardless of assumptions on $\Omega$.
Moreover, we do not know whether {\renewcommand{\betaM}{\beta_{ps}(\lambda_1(p))}\renewcommand{\alphaM}{\lambda_1(p)}\eqref{eq:D}} possesses a positive solution provided $\beta_{ps}(\lambda_1(p)) = \beta_*$, except in the case discussed in Theorem \ref{thm:nonexistence0}, cf.\ Theorem \ref{prop:betaps}.

\medskip
Theorem \ref{thm:1} in combination with the behaviour of $\E$ investigated in \cite{BobkovTanaka2017} allows us to state the following two  multiplicity results, among which 
Theorem \ref{thm:3multiple} is perhaps the most surprising since it indicates the occurrence of an $S$-shaped bifurcation diagram in the case $p>2q$, see Figure \ref{fig:1}. 
\begin{theorem}\label{thm:2multiple} 
Let $\alpha\in [\lambda_1(p), \alpha_*)$ and 
$\beta\in(\beta_*(\alpha),\beta_{ps}(\alpha))$, 
and assume \textnormal{\ref{H}} if $\alpha=\lambda_1(p)$. 
Then \eqref{eq:D} has at least two positive solutions $u_1$ and $u_2$ satisfying $\E(u_1)<0$ and $\E(u_2)<0$.
\end{theorem}

\begin{theorem}\label{thm:3multiple}
Assume \textnormal{\ref{H}}. 
Then for every $\beta\in (\beta_*,\beta_{ps}(\lambda_1(p))]$ 
there exists $\alpha_*(\beta) \in (0,\lambda_1(p))$ such that 
\eqref{eq:D} has at least three positive solutions 
for any $\alpha\in (\alpha_*(\beta),\lambda_1(p))$. 
\end{theorem} 

\begin{remark}
	All three solutions obtained in Theorem \ref{thm:3multiple} have negative energy, see Proposition \ref{prop:uniq}.
\end{remark}

Let us recall that if $\alpha \leq 0$, then the positive solution of \eqref{eq:D} is unique (see Proposition \ref{prop:uniq}), and it was unclear (see \cite[Remark 1]{BobkovTanaka2017}) whether a difficulty to extend the uniqueness to $\alpha \in (0, \lambda_1(p))$ lies only in the limitation of the method of the original proof, or a multiplicity of positive solutions can actually occur. Theorem \ref{thm:3multiple} answers this question in a nontrivial way.
We emphasize that this multiplicity result does not depend on the domain, as it happens, e.g., in the case of superlinear problems of the type $-\Delta_q u = |u|^{p-2} u$, cf.~\cite{Nazarov}. Moreover, there is no simple \textit{a priori} intuition about such multiplicity based on the behaviour of fiber functions of $\E$, since these functions have at most one critical point which is the point of global minimum, see Section \ref{sec:auxiliaryI}.
Finally, let us mention that the $S$-shaped bifurcation diagram indicated by Theorem \ref{thm:3multiple} clarifies the shape of the bifurcation diagrams (A) or (B) in \cite{KTT} obtained for the one-dimensional version of {\renewcommand{\betaM}{\lambda}\renewcommand{\alphaM}{\lambda}\eqref{eq:D}}.
See Figure \ref{fig:3} for some numerical results in the one-dimensional case.

\begin{remark}
	Properties of the family of critical points $\alpha_*(\beta)$, such as the continuity, monotonicity, etc., are mostly unknown. 
	We anticipate that the set of parameters $\alpha$, $\beta$ corresponding to the existence of three positive solutions of \eqref{eq:D} can be extended to a larger region as depicted by the dashed line on Figure \ref{fig:1}.
\end{remark}

\begin{figure}[ht]
	\centering
	\includegraphics[width=0.5\linewidth]{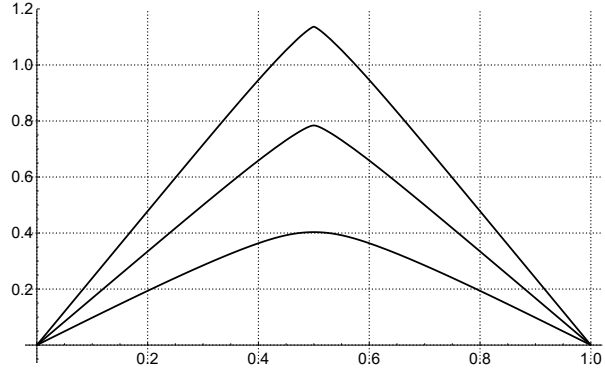}\\
	\caption{
		Three positive solutions of the one-dimensional problem \eqref{eq:D} on the interval $(0,1)$ 
		with $q=2$, $p=6$, $\alpha=\lambda_1(p)-0.1$ and $\beta=\beta_*+0.1$, found by the shooting method; consult with \cite[Appendix A]{BobkovTanaka2016} for explicit formulas for $\lambda_1(p)$ and $\beta_*$.
	}
	\label{fig:3}
\end{figure}

\medskip
The rest of the article is structured as follows.
In Section \ref{sec:auxiliaryI}, we introduce a few additional notations and provide an auxiliary lemma needed for the proof of Theorem \ref{thm:1} which we establish in Section \ref{sec:proof:Thm1}. 
Section \ref{sec:mountainpass} provides auxiliary results needed to prove Theorems \ref{thm:2multiple} and \ref{thm:3multiple}. These theorems are established in Section \ref{sec:proof:multiplicity}. 
Appendix \ref{appendix:WvcC} contains a ``$\W$ versus $C_0^1$ local minimizers''-type result for general problems with the $(p,q)$-Laplacian, which we also apply in Section \ref{sec:proof:multiplicity}.


\section{Auxiliary results I. The fibered functional \texorpdfstring{$\J$}{Jab}}\label{sec:auxiliaryI}
Take any $u \in \W$ satisfying $H_\alpha(u) \cdot G_\beta(u) < 0$ and consider the fiber function $t \mapsto \E(tu)$ for $t > 0$. It is not hard to observe that this function has a unique critical point $t_{\alpha,\beta}(u)$ given by 
\begin{equation}\label{tu}
t_{\alpha,\beta}(u) 
= \left(\frac{-G_\beta(u)}{H_\alpha(u)}\right)^{\frac{1}{p-q}} = 
\frac{|G_\beta(u)|^{\frac{1}{p-q}}}{|H_\alpha(u)|^{\frac{1}{p-q}}},
\end{equation}
see \cite[Proposition 6]{BobkovTanaka2015}.
Moreover, 
\begin{equation}\label{def:J}
\J(u) := \E(t_{\alpha,\beta}(u) u) = -\text{sign}(H_\alpha(u))\, 
\frac{p-q}{pq}\, \frac{|G_\beta(u)|^{\frac{p}{p-q}}}{|H_\alpha(u)|^{\frac{q}{p-q}}}. 
\end{equation}
In particular, if $G_\beta(u)<0<H_\alpha(u)$, then $t_{\alpha,\beta}(u)$ is the point of global minimum of the function $t \mapsto \E(tu)$, i.e.,  
\begin{equation}\label{eq:minivalue}
\min_{t>0}\E(tu) 
=
\E(t_{\alpha,\beta}(u)u)
\equiv 
\J(u)=-\frac{p-q}{pq}\, \frac{|G_\beta(u)|^{\frac{p}{p-q}}}{H_\alpha(u)^{\frac{q}{p-q}}}<0. 
\end{equation}
The functional $\J$ is called fibered functional \cite{pohozaev}, it is $0$-homogeneous, and if $u$ is a critical point of $\J$ satisfying $H_\alpha(u) \cdot G_\beta(u) < 0$, then $t_{\alpha,\beta}(u)u$ is a critical point of $\E$.

By $\N$ we denote the Nehari manifold associated to $\E$, that is,
\begin{equation*}\label{def:Nehari} 
\N = 
\left\{ 
v\in\W\setminus\{0\}:~
\langle \E^\prime(v),v \rangle=H_\alpha(v)+G_\beta(v)=0
\right\}.
\end{equation*}
Clearly, this set contains all nonzero critical points of $\E$. 
Notice that if we take any $u \in \W$ satisfying $H_\alpha(u) \cdot G_\beta(u) < 0$, then $t_{\alpha,\beta}(u) u \in \N$, see \cite[Proposition 10]{BobkovTanaka2017}. 

The following auxiliary result will be used in the proof of Theorem \ref{thm:1} in Section \ref{sec:proof:Thm1} below.
\begin{lemma}\label{lem:conv} 
	Let $\alpha \in [\lambda_1(p),\alpha_*)$, $\{\beta_n\} \subset [\beta_*(\alpha),+\infty)$ be a sequence converging to $\beta \geq \beta_*(\alpha)$, and $\mu \in (\lambda_1(q),\beta_*(\alpha))$. 
	Assume that a sequence $\{w_n\} \subset \W$ 
	satisfies $\|\nabla w_n\|_p=1$ for all $n \in \mathbb{N}$, and let $w_0 \in \W$ be such that $\{w_n\}$
	converges weakly in $\W$ 
	and strongly in $L^p(\Omega)$ to $w_0$ as $n \to +\infty$. 
	Assume, moreover, that 
	\begin{equation}\label{eq:assumseq}
	G_\mu(w_n) \leq 0 < H_\alpha(w_n)  
	\quad \text{for all}~ n \in \mathbb{N}, 
	\quad \text{and}\quad 
	-\infty \leq \liminf_{n \to +\infty}J_{\alpha,\beta_n}(w_n)<0. 
	\end{equation}
	Then $w_0 \not\equiv 0$ in $\Omega$, and we have
	$$
	G_\beta(w_0) < G_\mu(w_0) \leq 0 < H_\alpha(w_0)
	\quad \text{and}\quad 
	-\infty < \J(w_0)\le \liminf_{n\to +\infty}J_{\alpha,\beta_n}(w_n). 
	$$
\end{lemma} 
\begin{proof} 
	First we show that $w_0 \not\equiv 0$ in $\Omega$.
	Suppose, by contradiction, that $w_0 \equiv 0$ in $\Omega$. 
	That is, $\|w_n\|_p \to 0$ and $\|w_n\|_q \to 0$ as $n \to +\infty$.
	In particular, we have $H_\alpha(w_n) = 1 - o(1)$. 
	Moreover, since $G_{\mu}(w_n) \leq 0$, we see that $\|\nabla w_n\|_q \to 0$, which yields $G_{\beta_n}(w_n) \to 0$, and, consequently, $J_{\alpha,\beta_n}(w_n) \to 0$. 
	However, this is a contradiction to \eqref{eq:assumseq}, and hence $w_0 \not\equiv 0$ in $\Omega$. 
	
	By the weak lower semicontinuity, we readily get
	$$
	G_{\mu}(w_0) \leq \liminf_{n \to +\infty} G_{\mu}(w_n) \leq 0,
	$$
	which implies that
	$$
	\frac{\|\nabla w_0\|_q^q}{\|w_0\|_q^q} \leq \mu < \beta_*(\alpha). 
	$$
	Due to the definition \eqref{def:alpha_*} of $\beta_*(\alpha)$, we conclude that $H_\alpha(w_0)>0$.
	On the other hand, by our assumptions, we have $\beta > \mu$ and $\beta_n > \mu$ for all $n$, which implies that $G_\beta(w_0)<0$ and $G_{\beta_n}(w_n)<0$ for all $n$.
	Therefore, the weak lower semicontinuity of $G_{\beta_n}$ and $H_\alpha$ yields 
	$$
	-\infty < \J(w_0)\le \liminf_{n\to +\infty}J_{\alpha,\beta_n}(w_n),
	$$
	which completes the proof.
\end{proof}

\section{Beyond \texorpdfstring{$\beta_*(\alpha)$}{beta*(alpha)}. The proof of Theorem \ref{thm:1}}\label{sec:proof:Thm1} 
In this section, we prove Theorem \ref{thm:1}.
Throughout the section, we assume $\alpha\in [\lambda_1(p), \alpha_*)$ to be \textit{fixed}, and we require hypothesis \ref{H} if $\alpha=\lambda_1(p)$. 

The proof of Theorem \ref{thm:1} will rely on the consideration of the following minimization problem:
	\begin{equation}\label{eq:minimization}
	\mathcal{J}(\beta,\mu) 
	:=
	\inf
	\left\{\J(u):~ u \in \W,\ ~ G_{\mu}(u) < 0 < H_\alpha(u)
	\right\}, 
	\end{equation}
where we assume $\beta \ge \beta_*(\alpha)$ and $\mu \in (\lambda_1(q),\beta_*(\alpha)]$, and $\J$ is the fibered functional defined by \eqref{def:J}. 
Notice that the index of $G_\beta$ presented in $\J$ is, in general, different from the index of $G_\mu$ presented in the constraint. 
To the best of our knowledge, the idea of introduction of such constraints was originated in the work \cite{IK}.

Let us discuss several general properties of \eqref{eq:minimization}.
The admissible set for $\mathcal{J}(\beta,\mu)$ 
is nonempty because $\mu > \lambda_1(q)$ and $\alpha < \alpha_*$ yield $G_\mu(\varphi_q)<0<H_\alpha(\varphi_q)$. 
Consequently, we always have
\begin{equation}\label{eq:conv1} 
\mathcal{J}(\beta,\mu)\le \J(\varphi_q)<0,
\end{equation} 
since $\beta \geq \beta_*(\alpha) > \lambda_1(q)$.
If we let $\beta=\mu$, then
$\mathcal{J}(\beta,\beta)$ translates to the usual minimization problem of finding the {\it least energy solution} to $(D_{\alpha,\beta})$, see, e.g., \cite{BobkovTanaka2015,BobkovTanaka2017}. 
In particular, $\mathcal{J}(\beta_*(\alpha),\beta_*(\alpha))$ is attained, and if $u_*$ is a corresponding minimizer, 
then $t_{\alpha,\beta}(u_*)u_*$ is a solution 
of {\renewcommand{\betaM}{\beta_*(\alpha)}\eqref{eq:D}}, see Theorem \ref{thm:multi} in the case $\lambda_1(p)<\alpha<\alpha_*$ and Theorem \ref{thm:GM} in the case $\alpha=\lambda_1(p)$. 
Let us define
\begin{equation}\label{eq:mu0}
\mu_0 = \mu_0(\alpha) 
:=
\sup
\left\{ 
\frac{\|\nabla u_*\|_q^q}{\|u_*\|_q^q}:~
u_* ~\text{is a minimizer of}~ 
\mathcal{J}(\beta_*(\alpha),\beta_*(\alpha))
\right\}.
\end{equation}
\begin{proposition}\label{prop:2} 
	$\mu_0 < \beta_*(\alpha)$. 
\end{proposition}
\begin{proof} 
	It is clear that $\mu_0 \leq \beta_*(\alpha)$, since otherwise $G_{\beta_*(\alpha)}(u_*)>0$ for some minimizer $u_*$ of $\mathcal{J}(\beta_*(\alpha),\beta_*(\alpha))$, which is impossible, see \eqref{eq:minimization} with $\beta=\mu=\beta_*(\alpha)$. 
	Suppose, contrary to our claim, that $\mu_0 = \beta_*(\alpha)$. 
	That is, there exists a sequence of minimizers $\{u_k\}$ of $\mathcal{J}(\beta_*(\alpha),\beta_*(\alpha))$ such that $\frac{\|\nabla u_k\|_q^q}{\|u_k\|_q^q} \to \beta_*(\alpha)$. 
	Since $\J$ is 0-homogeneous, we may assume, without loss of generality, that $\|\nabla u_k\|_p=1$ for each $k$.
	Thus, the latter convergence yields $G_{\beta_*(\alpha)}(u_k) \to 0$.
	On the other hand, since $J_{\alpha,\beta_*(\alpha)}(u_k) = \mathcal{J}(\beta_*(\alpha),\beta_*(\alpha)) < 0$ by \eqref{eq:conv1}, 
we get from \eqref{eq:minivalue} that
	\begin{equation}\label{eq:m0<b}
	H_\alpha(u_k) = \left(\frac{p-q}{pq}\right)^\frac{p-q}{q} \frac{|G_{\beta_*(\alpha)}(u_k)|^\frac{p}{q} }{\left(-\mathcal{J}(\beta_*(\alpha),\beta_*(\alpha))\right)^\frac{p-q}{q}}
	\quad \text{for all}~ k \in \mathbb{N}.
	\end{equation}
	Substituting \eqref{eq:m0<b} into \eqref{tu}, we deduce, in view of the default assumption $p>q$, that 	$t_{\alpha,\beta_*(\alpha)}(u_k) \to +\infty$. 
	Moreover, by considering $|u_k|$ if necessary, we may assume that $u_k\ge 0$ in $\Omega$ for all $k$. 
	Recall now that $t_{\alpha,\beta_*(\alpha)}(u_k)u_k$ is a solution of {\renewcommand{\betaM}{\beta_*(\alpha)}\eqref{eq:D}}, and hence
	$$
	\left<H_\alpha'(u_k),\varphi\right> 
	+ t_{\alpha,\beta_*(\alpha)}(u_k)^{q-p}
	\left<G_{\beta_*(\alpha)}'(u_k),\varphi\right>
	= 0
	\quad \text{for all}~ \varphi \in \W,
	$$
	which implies that $u_k \to \varphi_p$ (strongly) in $\W$, up to a subsequence, and $\alpha=\lambda_1(p)$, see \cite[Lemma 3.3]{BobkovTanaka2016}.
	Therefore, if we fixed $\alpha > \lambda_1(p)$, then we get a contradiction, and, consequently, the proposition follows.
	Assume that we fixed $\alpha=\lambda_1(p)$. 
	Notice that in this case we require \ref{H}.
	Considering the $L^2$-orthogonal decomposition $u_k = \gamma_k \varphi_p + v_k$, where $\gamma_k = \|\varphi_p\|_2^{-2} \int_\Omega u_k \varphi_p \, dx$ and $\int_\Omega v_k \varphi_p \,dx = 0$, we see that $\gamma_k \to 1$ and $\|\nabla v_k\|_p \to 0$. 
	Employing now the improved Poincar\'e inequality from \cite{takac} along the same lines as in the proof of \cite[Proposition 11]{BobkovTanaka2017} (see, more precisely, \cite[pp.\ 1233-1234]{BobkovTanaka2017}), we deduce that $J_{\lambda_1(p),\beta_*}(u_k) \to 0$, which contradicts \eqref{eq:conv1}. Hence the proof is complete. 
\end{proof}

In general, if $\mathcal{J}(\beta,\mu)$ is attained, then the corresponding minimizer generates a solution of \eqref{eq:D}. 
We detail this fact as follows.
\begin{proposition}\label{prop:minimizer-localmini}
	Let $\beta \ge \beta_*(\alpha)$ and assume that $u_0\in \W$ is a minimizer of $\mathcal{J}(\beta,\mu)$ 
	for some $\mu \in (\lambda_1(q),\beta_*(\alpha)]$. 
	Then $t_{\alpha,\beta}(u_0)u_0$ is a local minimum point of $\E$ and
	$$
	\E(t_{\alpha,\beta}(u_0)u_0) \equiv \J(u_0)=\mathcal{J}(\beta,\mu)<0.
	$$ 
\end{proposition}
\begin{proof} 
	Suppose, by contradiction, that there exists a sequence $\{u_n\}$ convergent to $\tilde{u}_0:=t_{\alpha,\beta}(u_0)u_0$ in $\W$ such that 
	$$\E(u_n)<\E(\tilde{u}_0) 
	\quad \text{for all}~ n \in \mathbb{N}. 
	$$
	Using the fact that $u_0$ is a minimizer of $\mathcal{J}(\beta,\mu)$, we have  $G_\mu(\tilde{u}_0)<0<H_\alpha(\tilde{u}_0)$, and hence 
	$$
	G_\beta(u_n) \leq G_\mu (u_n)<0<H_\alpha(u_n)
	$$ 
	for all sufficiently large $n$, which means that 
	any such $u_n$ is an admissible function for $\mathcal{J}(\beta,\mu)$. 
	But then, using \eqref{eq:minivalue}, we get the following contradiction:
	$$
	\E(\tilde{u}_0) = \J(u_0)=\mathcal{J}(\beta,\mu)
	\le \J(u_n)
	=
	\E(t_{\alpha,\beta}(u_n)u_n) \leq \E(u_n) < \E(\tilde{u}_0)
	$$
	for all sufficiently large $n$. 
\end{proof} 

Let us now discuss the existence of a minimizer of $\mathcal{J}(\beta,\mu)$ required in Proposition \ref{prop:minimizer-localmini}.
\begin{lemma}\label{lem:3} 
	Let $\beta \geq \beta_*(\alpha)$ and $\mu \in (\lambda_1(q),\beta_*(\alpha))$.
	Then there exists a nonnegative function $u_0 \in \W$ satisfying $\|\nabla u_0\|_p = 1$ such that
	\begin{equation}\label{eq:lem3-1} 
	G_\mu(u_0) \leq 0 < H_\alpha(u_0)
	\quad \text{and}\quad 
	\J(u_0)\le \mathcal{J}(\beta,\mu)<0. 
	\end{equation}
\end{lemma} 
\begin{proof} 
	First, we recall that $\mathcal{J}(\beta,\mu) < 0$ by \eqref{eq:conv1}. 
	Let $\{u_n\}$ be a minimizing sequence for $\mathcal{J}(\beta,\mu)$. 
	Since $\J$ is $0$-homogeneous and even, we can assume, without loss of generality, that $\|\nabla u_n\|_p=1$ and $u_n \ge 0$ in $\Omega$ for all $n \in \mathbb{N}$ by considering $|u_n|$ if necessary. 
	Therefore, there exists a nonnegative function $u_0 \in \W$ such that $u_n \rightharpoonup u_0$ in $\W$ and $u_n \to u_0$ in $L^p(\Omega)$, up to an appropriate subsequence.
	Applying Lemma \ref{lem:conv} (with $\beta_n=\beta$), we get \eqref{eq:lem3-1}. Using again the $0$-homogeneity of $\J$, we can assume that $\|\nabla u_0\|_p = 1$, which completes the proof.
\end{proof} 

If the function $u_0$ obtained in Lemma \ref{lem:3} satisfies $G_\mu(u_0)<0$, then $u_0$ is a minimizer of $\mathcal{J}(\beta,\mu)$. Consequently, Proposition \ref{prop:minimizer-localmini} in combination with Remark \ref{rem:positive} imply that $t_{\alpha,\beta}(u_0) u_0$ is a positive solution of \eqref{eq:D}. 
Thus, the proof of Theorem \ref{thm:1} reduces to the search of such $\beta > \beta_*(\alpha)$ and $\mu \in (\lambda_1(q),\beta_*(\alpha))$ that $G_\mu(u_0)<0$. 
The details are as follows.

\begin{proof}[Proof of Theorem \ref{thm:1}]
	Let us fix any $\mu \in (\mu_0, \beta_*(\alpha))$, where $\mu_0$ is defined in \eqref{eq:mu0} and $\mu_0 < \beta_*(\alpha)$ by Proposition \ref{prop:2}. 
	Denote by $u_0 = u_0(\beta)$ a normalized nonnegative function given by Lemma \ref{lem:3}.
	We are going to obtain the existence of $\tilde{\beta}(\alpha)>\beta_*(\alpha)$ 
	such that $G_{\mu}(u_0(\beta)) < 0$ for any $\beta \in (\beta_*(\alpha),\tilde{\beta}(\alpha))$.
	Suppose, contrary to our claim, that there exists a  sequence 
	$\{\beta_n\}$ such that $\beta_n \searrow \beta_*(\alpha)$ and $G_\mu(u_0(\beta_n))=0$ for all $n \in \mathbb{N}$.
	We will reach a contradiction by showing that the corresponding sequence $\{u_0(\beta_n)\}$ converges in $\W$, up to a subsequence, to a minimizer of $\mathcal{J}(\beta_*(\alpha),\beta_*(\alpha))$, which is impossible in view of the definition of $\mu_0$.	
	
	Since $\|\nabla u_0(\beta_n)\|_p=1$ for all $n$, there exists $\bar{u} \ge 0$ such that 
	$u_0(\beta_n)$ converges to $\bar{u}$ weakly in $\W$ 
	and strongly in $L^p(\Omega)$ and $L^q(\Omega)$, up to an appropriate subsequence.
	At the same time, in view of \eqref{eq:lem3-1}, we have
	$$
	G_\mu(u_0(\beta_n))
	=
	0
	<
	H_\alpha(u_0(\beta_n))
	\quad \text{for all}~ n \in \mathbb{N},
	$$
	and
	\begin{equation}\label{eq:thm1proof3}
	-\infty \leq 
	\liminf_{n\to+\infty}J_{\alpha,\beta_n}(u_0(\beta_n))
	\le 
	\liminf_{n\to+\infty}\mathcal{J}(\beta_n,\mu) < 0,
	\end{equation}
	where the last inequality in  \eqref{eq:thm1proof3} follows from the uniform bound \eqref{eq:conv1}.
	Consequently, applying Lemma \ref{lem:conv} (with $\beta=\beta_*(\alpha)$) to the sequence $\{u_0(\beta_n)\}$, we deduce that
	\begin{equation}\label{eq:lem4-2}
	G_{\beta^*(\alpha)}(\bar{u}) 
	<
	G_\mu(\bar{u}) 
	\leq 
	0
	<
	H_\alpha(\bar{u})
	\end{equation}
	and
	\begin{equation}\label{eq:lem4-3}
	\mathcal{J}(\beta_*(\alpha),\beta_*(\alpha))
	\le J_{\alpha,\beta_*(\alpha)}(\bar{u})\le 
\liminf_{n\to+\infty}J_{\alpha,\beta_n}(u_0(\beta_n)) \le 
	\liminf_{n\to+\infty}\mathcal{J}(\beta_n,\mu),
	\end{equation}
	where the first inequality in \eqref{eq:lem4-3} follows from the fact that $\bar{u}$ is an admissible function for $\mathcal{J}(\beta_*(\alpha),\beta_*(\alpha))$, see \eqref{eq:lem4-2}.
	On the other hand, noting that any minimizer $u_*$ of $\mathcal{J}(\beta_*(\alpha),\beta_*(\alpha))$ 
	satisfies 
	$$
	G_{\beta_n}(u_*) < G_\mu(u_*) < 0 < H_\alpha(u_*)
	$$ by the definition of $\mu_0$ and the fact that $\beta_n > \beta_*(\alpha) >\mu>\mu_0$, we get 
	\begin{equation}\label{eq:thm1proof2}
	\mathcal{J}(\beta_n,\mu)\le 
	J_{\alpha,\beta_n}(u_*)=J_{\alpha,\beta_*(\alpha)}(u_*)+o(1) 
	=\mathcal{J}(\beta_*(\alpha),\beta_*(\alpha))+o(1). 
	\end{equation}
	Therefore, combining \eqref{eq:lem4-3} with \eqref{eq:thm1proof2}, we conclude that  $\mathcal{J}(\beta_*(\alpha),\beta_*(\alpha)) 
	=J_{\alpha,\beta_*(\alpha)}(\bar{u})$, which means that $\bar{u}$ is a minimizer of 
	$\mathcal{J}(\beta_*(\alpha),\beta_*(\alpha))$.
	Moreover, since \eqref{eq:lem4-3} and \eqref{eq:thm1proof2} also imply
\begin{equation}\label{eq:thm1proof0}
	J_{\alpha,\beta_*(\alpha)}(\bar{u})=
\liminf_{n\to+\infty}J_{\alpha,\beta_n}(u_0(\beta_n)), 
\end{equation}
we get $u_0(\beta_n) \to \bar{u}$ in $\W$, up to a subsequence. 
Indeed, 
if we suppose that there is no strong convergence, then $\|\nabla \bar{u}\|_p<\liminf\limits_{n\to +\infty}\|\nabla u_0(\beta_n)\|_p$, and hence $(0<)H_\alpha(\bar{u}) <\liminf\limits_{n\to+\infty} H_\alpha(u_0(\beta_n))$, which implies a contradiction to the equality in \eqref{eq:thm1proof0}. 
Finally, let us notice that the strong convergence of $\{u_0(\beta_n)\}$ gives $G_\mu(\bar{u})=0$, see \eqref{eq:lem4-2}. 
However, this contradicts the definition of $\mu_0$ and the fact that $\mu > \mu_0$.
\end{proof}

\section{Auxiliary results II. Mountain pass type arguments}\label{sec:mountainpass}
In this section, we prepare several results related to the mountain pass theorem, which will be used to prove Theorems \ref{thm:2multiple} and \ref{thm:3multiple} in Section \ref{sec:proof:multiplicity} below. 
Since our aim is to find \textit{positive} solutions of \eqref{eq:D}, in the arguments of this section it will be convenient to consider the $C^1$-functional 
$$
\widetilde{E}_{\alpha,\beta}(u) 
=
\frac{1}{p}\widetilde{H}_\alpha (u)
+\frac{1}{q} \widetilde{G}_\beta (u),  \quad u \in \W,  
$$
where 
$$
\widetilde{H}_\alpha (u):=\|\nabla u\|_p^p-\alpha\|u_+\|_p^p 
\quad {\rm and} \quad 
\widetilde{G}_\beta(u):=\|\nabla u\|_q^q-\beta\|u_+\|_q^q,
$$
and $u_+ := \max\{u,0\}$. 
The functional $\Ep$ differs from $\E$ in that if $u \in \W$ is an \textit{arbitrary} critical point of $\Ep$, then $u$ is a nonnegative solution of $\eqref{eq:D}$, which can be easily seen by taking $u_- := \max\{-u,0\}$ as a test function.
Moreover, $u$ is a positive solution belonging to ${\rm int}\,C_0^1(\overline{\Omega})_+$ 
provided $u\not\equiv 0$, see Remark \ref{rem:positive}. 

Now we discuss the assumptions under which $\Ep$ satisfies the Palais--Smale condition. 
\begin{lemma}\label{lem:PS} 
Let $(\alpha,\beta)\not=(\lambda_1(p),\beta_*)$. 
Then $\Ep$ satisfies the Palais--Smale condition. 
\end{lemma} 
\begin{proof} 
	Let us take any Palais--Smale sequence $\{u_n\}$ for $\Ep$. 
	According to the $(S_+)$-property of the operator $-\Delta_p - \Delta_q$ (see, e.g., \cite[Remark 3.5]{BobkovTanaka2016}), the desired Palais--Smale condition for $\Ep$ will follow if $\{u_n\}$ is bounded. 
	Suppose, by contradiction, that $\|\nabla u_n\|_p \to +\infty$ as $n \to +\infty$, up to a subsequence. 
	Considering the sequence of normalized functions $v_n:=\frac{u_n}{\|\nabla u_n\|_p}$ and 
	arguing in much the same way as in \cite[Lemma 3.3]{BobkovTanaka2016}, we derive that 
	$\{v_n\}$ converges in $\W$, up to a subsequence, to some nonzero eigenfunction $v_0$ of the $p$-Laplacian associated to the eigenvalue $\alpha$.
	Noting that 
	$$
	o(1)\|\nabla (u_n)_-\|_p
	=
	\left< \Ep^\prime (u_n),-(u_n)_-\right>
	=
	\|\nabla (u_n)_-\|_p^p + \|\nabla (u_n)_-\|_q^q, 
	$$
	we deduce that $v_0\ge 0$ in $\Omega$. This yields $\alpha=\lambda_1(p)$ and $v_0=\varphi_p$, since $\varphi_p$ is the only constant-sign eigenfunction of the $p$-Laplacian and we assumed that $\|\nabla\varphi_p\|_p=1$. 
	Thus, if $\alpha \neq \lambda_1(p)$, then we get a contradiction, and hence the Palais--Smale condition for $\Ep$ holds for any $\beta \in \mathbb{R}$. 
	On the other hand, in the case $\alpha=\lambda_1(p)$ and $\beta \neq \beta_*$, we get 
	\begin{align*} 
	o(1) =\frac{1}{\|\nabla u_n\|_p^q}
	\left(p\Ep(u_n) -\left< \Ep^\prime (u_n),u_n\right>\right) =\left(\frac{p}{q}-1\right)\,\widetilde{G}_\beta (v_n).
	\end{align*}
	This yields $0=\widetilde{G}_\beta (\varphi_p)=G_\beta (\varphi_p)$, which contradicts the assumption $\beta \neq \beta_*$.
\end{proof}

Before providing a mountain pass-type result, we give the following auxiliary lemma. 
\begin{lemma}\label{lem:construct-path}
Let $u_1$ be a local minimum point of $\Ep$ such that 
\begin{equation}\label{eq:path-1}
\inf_{u\in \mathcal{N}_{\alpha,\beta}}\E(u) < \Ep(u_1)<0.
\end{equation}
Then there exists a continuous path $\eta\in C([0,1],\W)$ 
such that 
\begin{equation}\label{eq:path-2} 
\eta(0)=u_1,
\quad 
\Ep(\eta(1)) < \Ep(u_1), 
\quad \text{and} \quad 
\max_{s\in[0,1]}\Ep(\eta(s))<0. 
\end{equation}
\end{lemma}
\begin{proof} 
Noting that $u_1 \in {\rm int}\,C_0^1(\overline{\Omega})_+$ (see Remark \ref{rem:positive}) and $u_1 \in \mathcal{N}_{\alpha,\beta}$, we get
$$
0>\Ep(u_1)=\E(u_1)=\frac{p-q}{pq}G_\beta(u_1)
=-\frac{p-q}{pq}H_\alpha(u_1),
$$
and hence 
\begin{equation*}\label{eq:path-3}
\widetilde{G}_\beta(u_1) 
=
G_\beta(u_1)
< 0 <
H_\alpha(u_1)
=
\widetilde{H}_\alpha(u_1). 
\end{equation*} 
According to \eqref{eq:path-1}, we can find $v_1 \in\mathcal{N}_{\alpha,\beta}$ such that $\E(v_1)<\Ep(u_1)<0$.
Moreover, since $H_\alpha$ and $G_\beta$ are even, we may assume, 
by considering $|v_1|$ if necessary, that 
$v_1 \geq 0$ in $\Omega$.
Therefore,
\begin{equation*}\label{eq:path-5} 
\widetilde{G}_\beta(v_1)
=
G_\beta(v_1) 
< 0 <
H_\alpha(v_1)
=
\widetilde{H}_\alpha(v_1) \quad 
\text{and} \quad 
\Ep(v_1)<\Ep(u_1)<0.
\end{equation*}

Let us consider the path
$$
\xi(s) 
= 
\left((1-s)u_1^q +sv_1^q\right)^{1/q}
\quad \text{for}~ s\in [0,1]. 
$$
The hidden convexity of $\xi$ (see, e.g., \cite[Lemma 2.4]{TTU} or \cite[Proposition 2.6]{BF}) implies that
\begin{equation}\label{eq:path-6}
G_\beta(\xi(s))
\le 
(1-s)G_\beta (u_1)
+
s G_\beta(v_1)
\le 
\max\{G_\beta(u_1),G_\beta(v_1)\}<0
\quad \text{for all}~ s\in[0,1]. 
\end{equation}
Assume first that $H_\alpha(\xi(s))>0$ for all $s\in [0,1]$, and define the new path
$\eta(s) = t_{\alpha,\beta}(\xi(s))\xi(s)$, $s\in [0,1]$, 
where $t_{\alpha,\beta}(\xi(s))$ is given by \eqref{tu}. 
Noting that $t_{\alpha,\beta}(\xi(0))=t_{\alpha,\beta}(\xi(1))=1$ in view of $u_1,v_1 \in \N$, and that
$$
\Ep(\eta(s))
=
\E(\eta(s))
=
\J(\xi(s))
=
-\frac{p-q}{pq}\, 
\frac{|G_\beta(\xi(s))|^\frac{p}{p-q}}{H_\alpha(\xi(s))^\frac{q}{p-q}}<0
\quad \text{for all}~ s \in [0,1]
$$ 
by \eqref{eq:minivalue}, we readily see that $\eta$ satisfies \eqref{eq:path-2}.

Recalling that $H_\alpha(\xi(0))>0$ and $H_\alpha(\xi(1))>0$, assume now that there exists $s_0 \in (0,1)$ such that $H_\alpha(\xi(s_0))=0$. 
Without loss of generality, we may set 
$$
s_0 = \inf\{s\in (0,1):~ H_\alpha(\xi(s))\le 0\},
$$ 
and so $H_\alpha (\xi(s))>0$ for all $s\in (0,s_0)$.  
This implies that $\J(\xi(s))\to -\infty$ as $s \nearrow s_0$, thanks to \eqref{eq:path-6}. 
Thus, there exists some $s_1 \in (0, s_0)$ such that 
$$
\J(\xi(s_1)) <\Ep(u_1). 
$$
Considering the path $\eta(s) = t_{\alpha,\beta}(\xi(s_1 s))\xi(s_1 s)$ for $s\in[0,1]$, we complete the proof. 
\end{proof}

\begin{theorem}\label{thm:MPsol} 
Let $(\alpha,\beta)\not=(\lambda_1(p),\beta_*)$. 
Assume that $u_1$ is a local minimum point of $\Ep$ such that 
\begin{equation}\label{eq:path-4}
\inf_{u\in \mathcal{N}_{\alpha,\beta}}\E(u) < \Ep(u_1)<0.
\end{equation}
Then there exists another critical point $u_2$ of $\Ep$ satisfying
\begin{equation}\label{eq:MP-2}
\Ep(u_1) \le \Ep(u_2) < 0. 
\end{equation}
\end{theorem}
\begin{proof} 
	Let $\eta$ be a path given by Lemma \ref{lem:construct-path}.
	Since $u_1$ is a local minimum point of $\Ep$, there exists $r \in (0, \|\nabla (u_1 - \eta(1))\|_p)$ such that
	\begin{equation*}\label{eq:MP-3} 
	\Ep(u_1) \le \Ep(u) < 0 
	\quad \text{for every}~ u\in B_r(u_1),
	\end{equation*}
	where $B_r(u_1) = \{u \in \W: \|\nabla (u_1-u)\|_p \leq r\}$. 
	Therefore, the generalized mountain pass theorem \cite[Theorem 1]{PS} in combination with Lemma \ref{lem:PS} implies that 
	$$
	c := \inf_{\gamma\in\Gamma} \max_{s\in[0,1]} \Ep(\gamma(s)) \geq \Ep(u_1)
	$$
	is a critical level of $\Ep$, and there exists a critical point $u_2$ on the level $c$ which is different from $u_1$.
	Here
	$$
	\Gamma 
	:=
	\left\{
	\gamma\in C([0,1],\W):~ \gamma(0) = u_1,~ 
	\gamma(1)=\eta(1)
	\right\}.
	$$
	The properties \eqref{eq:path-2} of the admissible path $\eta$ yield $c<0$, which gives \eqref{eq:MP-2}.
\end{proof}

\section{Multiplicity. The proofs of Theorems \ref{thm:2multiple} and \ref{thm:3multiple}}\label{sec:proof:multiplicity}
In this section, we prove Theorems \ref{thm:2multiple} and \ref{thm:3multiple} using the results of Section \ref{sec:mountainpass}.
A local minimum point of $\Ep$ will be obtained by the super- and subsolution method.
Let us denote, for brevity,  
\begin{equation*}\label{def:f}
f_{\alpha,\beta}(u)
=
\alpha |u|^{p-2}u+\beta|u|^{q-2}u,
\end{equation*}
and recall that a function $u \in W^{1,p}$ is called supersolution 
(resp.\ subsolution) of \eqref{eq:D} if 
$u \ge 0$ (resp.\ $\le 0$) on $\partial\Omega$ in the sense of traces and 
\begin{equation}\label{eq:supersubsol}
\intO |\nabla u|^{p-2}\nabla u\nabla\varphi \,dx 
+ 
\intO |\nabla u|^{q-2}\nabla u\nabla\varphi \,dx 
\geq 
\intO f_{\alpha,\beta}(u) \varphi\,dx
\quad (\text{resp.}~\le 0)
\end{equation}
for any nonnegative $\varphi\in \W$. 
If, in addition, the strict inequality in \eqref{eq:supersubsol} is satisfied for any nonnegative and nonzero $\varphi$, then 
$u$ is called strict supersolution (resp.\ strict subsolution) of \eqref{eq:D}.

Taking any $v, w \in L^\infty(\Omega)$ such that $v\le w$ a.e.\ in $\Omega$, 
we introduce the truncation
\begin{gather*}
f_{\alpha,\beta}^{[v,w]}(x,t)
=
\begin{cases}
f_{\alpha,\beta}(v(x)) & {\rm if}\ t\le v(x),
\\
f_{\alpha,\beta}(t) & {\rm if }\ v(x) < t < w(x), \\
f_{\alpha,\beta}(w(x)) & {\rm if}\ t\ge w(x),
\end{cases}
\end{gather*}
and define the corresponding $C^1$-functional 
\begin{equation*}
\E^{[v,w]}(u)
=
\frac{1}{p}\intO |\nabla u|^p\, dx 
+
\frac{1}{q}\intO |\nabla u|^q\,dx 
-
\intO \int_0^{u(x)}f_{\alpha,\beta}^{[v,w]}(x,t)\,dt\,dx, 
\quad
u \in \W.
\end{equation*}
If $v$ and $w$ are sub- and supersolutions of \eqref{eq:D}, respectively, then critical points of $\E^{[v,w]}$ are solutions of \eqref{eq:D} and they belong to the ordered interval $[v,w]$, see, e.g., \cite[Remark 2]{BobkovTanaka2015}. 

We will make use of the following two lemmas.
\begin{lemma}[\protect{\cite[Lemma 6 and Remark 2]{BobkovTanaka2015}}]\label{lem:ss-method} 
Let $\alpha \in \mathbb{R}$ and $\beta>\lambda_1(q)$, and let 
$w\in {\rm int}\,C^1(\overline{\Omega})_+$ be a positive supersolution of 
\eqref{eq:D}. 
Then $\inf_{\W} \E^{[0,w]}<0$, the infimum is attained, and the corresponding global minimum point $u \in [0,w]$ satisfies \eqref{eq:D} and belongs to ${\rm int}\,C_0^1(\overline{\Omega})_+$. 
\end{lemma} 

\begin{lemma}\label{lem:minc}
	Let $\alpha \geq 0$ and $\beta > \lambda_1(q)$. 
	Let $w\in {\rm int}\,C_0^1(\overline{\Omega})_+$ be a positive strict supersolution of \eqref{eq:D}, that is,
\begin{equation}\label{eq:minc:0} 
\langle \E^\prime (w),\varphi \rangle >0 \quad 
\text{for any nonnegative and nonzero}~ \varphi\in \W.
\end{equation}
	Let $u \in {\rm int}\,C_0^1(\overline{\Omega})_+$ be a global minimum point of $\E^{[0,w]}$ given by Lemma \ref{lem:ss-method}. 
	Then $u \in (0,w)$ and $u$ is a local minimum point of both $\E$ and 
$\widetilde{E}_{\alpha,\beta}$ in $C_0^1(\overline{\Omega})$-topology.
\end{lemma}
\begin{proof}
	Noting that $u \in (0,w]$ and 
$\frac{\partial u}{\partial \nu}, \frac{\partial w}{\partial \nu}<0$ on 
$\partial \Omega$, 
we will prove that 
\begin{equation}\label{eq:minc:1}
u<w ~{\rm in}~ \Omega
 \quad {\rm and}\quad 
\frac{\partial u}{\partial \nu}>
\frac{\partial w}{\partial \nu} ~{\rm on}~ \partial \Omega. 
\end{equation} 
This fact directly implies the desired results. 
Indeed, if \eqref{eq:minc:1} holds, then $w-u \in {\rm int}\,C_0^1(\overline{\Omega})_+$ and hence, taking a sufficiently small $\kappa>0$, we get
$u+v\in [0,w]$ for any $v \in C_0^1(\overline{\Omega})$ satisfying  $\|v\|_{C_0^1(\overline{\Omega})}<\kappa$, 
whence 
$$
\E(u)=\widetilde{E}_{\alpha,\beta}(u)
=
\E^{[0,w]}(u)
=
\inf_{\W} \E^{[0,w]} 
\leq
\E^{[0,w]}(u+v)
=
\widetilde{E}_{\alpha,\beta}(u+v)
=
\E(u+v).
$$
	To establish \eqref{eq:minc:1}, we first show that $u<w$ in a neighbourhood of $\partial \Omega$, and then we derive that $u<w$ in the remaining part of $\Omega$. The details are as follows.

	For a sufficiently small $\delta>0$, we define
	$$
	\Omega_\delta 
	=
	\left\{
	x \in \Omega:~ \text{dist}(x,\partial \Omega) < \delta 
	\right\}.
	$$
	Since $u, w \in {\rm int}\,C_0^1(\overline{\Omega})_+$, one can find $\varepsilon, \delta > 0$ such that $|\nabla \left((1-s) u + s w \right)| > \varepsilon$ in $\overline{\Omega}_\delta$ for all $s \in [0,1]$.
	Indeed, suppose, by contradiction, that for any $n \in \mathbb{N}$ there exist $x_n \in \overline{\Omega}_{1/n}$ and $s_n \in [0,1]$ such that
	$|\nabla \left((1-s_n) u(x_n) + s_n w(x_n)\right)| \leq \frac{1}{n}$. 
	Passing to appropriate subsequences, we get
	$x_n \to x_0 \in \partial \Omega$, $s_n \to s_0 \in [0,1]$, and $|\nabla \left((1-s_0) u(x_0) + s_0 w(x_0)\right)| = 0$. 
	However, this contradicts the fact that 
	$\frac{\partial u}{\partial\nu}(x_0)$, $\frac{\partial w}{\partial\nu}(x_0) < 0$.

Let us denote, for short, 
$$
\mathcal{A}(\mathbf{a}) := |\mathbf{a}|^{p-2}\mathbf{a} 
+|\mathbf{a}|^{q-2}\mathbf{a} 
\quad {\rm for}\ \mathbf{a} \in \mathbb{R}^N, 
$$
and define the linealization $N\times N$-matrix $A(\mathbf{a})$ as
	\begin{align*}
	A(\mathbf{a})
	= |\mathbf{a}|^{p-2}
	\left[
	I + (p-2) \frac{\mathbf{a}
		\otimes
		\mathbf{a}
	}{|\mathbf{a}|^2}
	\right]
	+
	|\mathbf{a}|^{q-2}
	\left[
	I + (q-2) \frac{\mathbf{a}
		\otimes
		\mathbf{a}
	}{|\mathbf{a}|^2}
	\right]
	\quad {\rm for}\ \mathbf{a} \in \mathbb{R}^N \setminus \{0\}, 
	\end{align*}
	where $I$ is the identity matrix and  $\otimes$ denotes the Kronecker product, see, e.g., \cite[Appendix A.2]{PTT}.
Consider now $v := w-u$. 
	Clearly, $v \geq 0$ in $\Omega$. 
	Recalling that $w$ is a strict supersolution of \eqref{eq:D}, we subtract \eqref{eq:D:weak} from \eqref{eq:minc:0} and deduce, according to the mean value theorem, that $v$ satisfies 
	\begin{align}
	-\text{div}\left(\left[\int_0^1 A\left(\nabla \left((1-s) u + s w\right)\right) ds \right] \nabla v\right) 
& =-\text{div}(\mathcal{A}(\nabla w)-\mathcal{A}(\nabla u)) 
\nonumber \\ 
&> f_{\alpha,\beta}(w) - f_{\alpha,\beta}(u) \ge 0
	\quad \text{in}~ \Omega_\delta, \label{eq:v}
	\end{align}
	in the weak sense, 
	where the last inequality follows from the fact that $\alpha$ and $\beta$ are nonnegative and $w \geq u > 0$ in $\Omega$. 
	Applying the estimates \cite[(A.10)]{PTT} to the matrix $A$, we get
	\begin{align}
	\notag
	\big(
	\min\{1,p-1\} |\mathbf{a}|^{p-2}
	&+
	\min\{1,q-1\} |\mathbf{a}|^{q-2}
	\big)
	|\xi|^2	
	\leq 	
	\left<A(\mathbf{a}) \xi, \xi \right>_{\mathbb{R}^N}
	\\\label{eq:bounds2}
	&\leq
	\big(
	\max\{1,p-1\} |\mathbf{a}|^{p-2}
	+
	\max\{1,q-1\} |\mathbf{a}|^{q-2}
	\big)
	|\xi|^2
	\end{align}
	for any $\xi \in \mathbb{R}^N$ and $\mathbf{a} \in \mathbb{R}^N \setminus \{0\}$.
	Here, for clarity, we denote by $\left<\cdot,\cdot\right>_{\mathbb{R}^N}$ the usual scalar product in $\mathbb{R}^N$.
	Recalling that $|\nabla \left((1-s) u + s w\right)| > \varepsilon$ in $\overline{\Omega}_\delta$ for all $s \in [0,1]$, we employ the inequalities \cite[(A.4) and (A.6)]{PTT} to see that  for any $r>1$ there exist $C_1,C_2>0$ such that
	\begin{align}
	\notag
	C_1 \left(\max_{s \in [0,1]} |\nabla((1-s) u + s w)|
	\right)^{r-2} 
	&\leq 
	\int_0^1 |\nabla((1-s) u + s w)|^{r-2} \, ds
	\\\label{eq:bounds}
	&\leq 
	C_2 \left(\max_{s \in [0,1]} |\nabla((1-s) u + s w)|
	\right)^{r-2}
	\quad \text{in}~ \overline{\Omega}_\delta.
	\end{align}
	Thus, taking $\mathbf{a} = \nabla((1-s) u + s w)$ in \eqref{eq:bounds2} and using \eqref{eq:bounds} with $r=p$ and $r=q$, we conclude that there exist $C_3, C_4 > 0$ satisfying
	\begin{align*}
	C_3 |\xi|^2 \leq \left<\left[\int_0^1 A\left(\nabla \left((1-s) u + s w\right)\right) ds \right] \xi, \xi \right>_{\mathbb{R}^N}
	&\leq 
	C_4 |\xi|^2
	\quad 
	\text{in}~ \overline{\Omega}_\delta,~ 
	\text{for any}~ \xi \in \mathbb{R}^N.
	\end{align*}
	That is, the differential operator in \eqref{eq:v} is uniformly elliptic in $\overline{\Omega}_\delta$.	
	Therefore, in view of the strict inequality in \eqref{eq:v}, the strong maximum principle yields $v > 0$ in $\Omega_\delta$ and $\frac{\partial v}{\partial \nu}<0$ on $\partial \Omega$. Consequently, $u < w$ in $\Omega_\delta$ and $\frac{\partial u}{\partial \nu}>
	\frac{\partial w}{\partial \nu}$ on $\partial \Omega$.
	
	Let us now fix some $\delta' \in (0, \delta)$ and a sufficiently small $C>0$ such that $u+C \leq w$ on $\partial \Omega_{\delta'} \cap \Omega$. 
	Denoting $z = u+C$, we see that 
	\begin{equation}\label{eq:z}
	\intO |\nabla z|^{p-2}\nabla z\nabla\varphi \,dx 
	+ 
	\intO |\nabla z|^{q-2}\nabla z\nabla\varphi \,dx 
	= 
	\intO f_{\alpha,\beta}(u) \varphi\,dx
	\quad \text{for any}~ \varphi \in \W.
	\end{equation}
	Therefore, subtracting \eqref{eq:minc:0} from \eqref{eq:z} and taking $\varphi = \max\{z-w,0\}$ in $\Omega \setminus \Omega_{\delta'}$ and $\varphi=0$ in $\Omega_{\delta'}$, we derive that 
	\begin{align*}
0\le 	&\int_{\{z>w\} \cap (\Omega \setminus \Omega_{\delta'})} \left(|\nabla z|^{p-2}\nabla z - |\nabla w|^{p-2}\nabla w \right) (\nabla z - \nabla w) \,dx 
	\\
	&+ 
	\int_{\{z>w\} \cap (\Omega \setminus \Omega_{\delta'})} \left(|\nabla z|^{q-2}\nabla z - |\nabla w|^{q-2}\nabla w \right) (\nabla z - \nabla w) \,dx 
\\
& 	\le \int_{\{z>w\} \cap (\Omega \setminus \Omega_{\delta'})} 
\left(f_{\alpha,\beta}(u) - f_{\alpha,\beta}(w)\right)(z-w)\,dx
	\le 0,
	\end{align*}
	which implies that $\{z>w\} = \emptyset$ in $\Omega \setminus \Omega_{\delta'}$. 
	Thus, $u+C \leq w$ and, consequently, $u < w$ in $\Omega \setminus \Omega_{\delta'}$. Recalling that $u < w$ in $\Omega_{\delta}$, we conclude that $u \in (0,w)$ in $\Omega$. 
	Thus, \eqref{eq:minc:1} is satisfied, which completes the proof.
\end{proof}

\subsection{Proof of Theorem \ref{thm:2multiple}}

Fix any $\alpha\in [\lambda_1(p),\alpha_*)$ and 
$\beta \in (\beta_*(\alpha), \beta_{ps}(\alpha))$. 
Choosing an arbitrary $\beta^\prime\in (\beta,\beta_{ps}(\alpha)]$, we denote by $w \in {\rm int}\,C_0^1(\overline{\Omega})_+$ a positive solution of {\renewcommand{\betaM}{\beta^\prime}\eqref{eq:D}}, see Theorem \ref{prop:betaps} in the case $\alpha>\lambda_1(p)$ and Theorem \ref{thm:summary} in the case $\alpha=\lambda_1(p)$ for the existence result. 
Clearly, $w$ is a strict supersolution of \eqref{eq:D}. 
Hence, thanks to Lemma \ref{lem:ss-method}, 
we can find a global minimum point $u_1 \in {\rm int}\,C_0^1(\overline{\Omega})_+$ of $E_{\alpha,\beta}^{[0,w]}$ 
such that $\E(u_1)=E_{\alpha,\beta}^{[0,w]}(u_1)<0$, and $u_1$ is a positive solution of \eqref{eq:D}. 
Moreover, according to Lemma \ref{lem:minc}, 
$u_1$ is a local minimum point of $\Ep$ in $C^1_0(\overline{\Omega})$-topology. 
Therefore, applying Theorem \ref{thm:minimizer} with $f(x,t)=\alpha t_+^{p-1}+\beta t_+^{q-1}$, we see that 
$u_1$ is a local minimum point of $\Ep$ in $\W$. 

On the other hand, it was shown in 
\cite[Theorem 2.5]{BobkovTanaka2017}
that $\inf_{u\in\N} \E(u)=-\infty$ provided $\beta>\beta_*(\alpha)$. 
Consequently, \eqref{eq:path-4} holds, whence Theorem \ref{thm:MPsol} yields the existence of the second positive solution $u_2$ of \eqref{eq:D} which satisfies \eqref{eq:MP-2}.
\qed

\subsection{Proof of Theorem \ref{thm:3multiple}}
Fix any $\beta \in (\beta_*, \beta_{ps}(\lambda_1(p))]$ and denote by $w \in {\rm int}\,C_0^1(\overline{\Omega})_+$ a positive solution of {\renewcommand\alphaM{\lambda_1(p)}\eqref{eq:D}}, see Theorem \ref{thm:summary} for the existence of $w$.
Evidently, $w$ is a strict supersolution of \eqref{eq:D} for any $\alpha \in [0, \lambda_1(p))$.
Therefore, arguing as in the proof of Theorem \ref{thm:2multiple} above, we can find a local minimum point $u_1 = u_1(\alpha) \in {\rm int}\,C_0^1(\overline{\Omega})_+$ of $\Ep$ in $\W$ such that
\begin{equation}\label{eq:3multi-0}
u_1(\alpha) \in (0,w)
\quad \text{and} \quad
\Ep(u_1(\alpha))=\E(u_1(\alpha)) < 0
\quad \text{for any}~ \alpha \in [0, \lambda_1(p)).
\end{equation}
Thus, $u_1(\alpha)$ is the first positive solution of \eqref{eq:D}. 
Moreover, in view of the uniform $L^\infty$-bound of $u_1(\alpha)$ in \eqref{eq:3multi-0}, we get
\begin{equation}\label{eq:3multi-00}
\inf\left\{\E(u_1(\alpha)):~
\alpha \in [0, \lambda_1(p))\right\} 
>-\infty. 
\end{equation}

Let $u_2 = u_2(\alpha) \in {\rm int}\,C_0^1(\overline{\Omega})_+$ be a global minimum point of $\E$ for $\alpha < \lambda_1(p)$ obtained in \cite[Proposition 1]{BobkovTanaka2017}. 
It is proved in \cite[Proposition 2 (i)]{BobkovTanaka2017} that 
\begin{equation}\label{eq:3multi-1}
\E(u_2(\alpha)) \to -\infty, 
\quad 
\|u_2(\alpha)\|_p \to +\infty,
\quad \text{and} \quad 
\frac{u_2(\alpha)}{\|u_2(\alpha)\|_p}
\to \frac{\varphi_p}{\|\varphi_p\|_p}
~\text{in}~ \W
\end{equation}
as $\alpha \nearrow \lambda_1(p)$. 
Comparing \eqref{eq:3multi-00} and \eqref{eq:3multi-1}, we derive the existence of $\alpha_*(\beta) \in [0, \lambda_1(p))$ such that
\begin{equation}\label{eq:3multi-2} 
\E(u_2(\alpha)) < \E(u_1(\alpha))
\quad \text{for any}~  
\alpha\in 
(\alpha_*(\beta),\lambda_1(p)).
\end{equation} 
Hence, $u_2(\alpha) \not= u_1(\alpha)$ whenever $\alpha\in (\alpha_*(\beta),\lambda_1(p))$. 
Moreover, we note that, in fact, $\alpha_*(\beta) \in (0, \lambda_1(p))$ due to the uniqueness result in Proposition \ref{prop:uniq}.
On the other hand, in view of \eqref{eq:3multi-2}, Theorem \ref{thm:MPsol} provides us with the existence of 
the third positive solution $u_3(\alpha)$ of \eqref{eq:D} for any $\alpha\in 
(\alpha_*(\beta),\lambda_1(p))$, and $u_3(\alpha)$ is different from $u_1(\alpha)$ and $u_2(\alpha)$. 
\qed

\appendix 
\section{\texorpdfstring{$\W$}{W01p} versus \texorpdfstring{$C_0^1$}{C01} local minimizers}\label{appendix:WvcC}

Let $f: \Omega\times\mathbb{R} \to \mathbb{R}$ be any Carath\'eodory function and let $F(x,u) = \int_0^u f(x,v) \, dv$ be the primitive of $f$.
Along this section, we assume that $f$ satisfies the following subcritical growth condition: 
\begin{enumerate}[label={(G)}]
	\item\label{G} There exist $C>0$ and $r \in [1,p^*)$ such that 
	$$
	|f(x,t)| \le C(1+|t|^{r-1}) 
	\quad \text{for every}~
	t\in\mathbb{R}
	~\text{and a.e.}~
	x\in\Omega, 
	$$
	where $p^*=\frac{pN}{N-p}$ if $N>p$, and $p^*=+\infty$ if $N\le p$.
\end{enumerate}

It is well known that the functional  
\begin{equation*}
I(u)
=
\frac{1}{p}\intO |\nabla u|^p\,dx
+
\frac{1}{q}\intO |\nabla u|^q\,dx -
\intO F(x,u)\,dx,
\quad u \in \W,
\end{equation*}
is weakly lower semicontinuous and of class $C^1$ under the assumption \ref{G}.

The following result can be obtained in much the same way as \cite[Theorem 1.2]{AMA} or \cite[Theorem 23]{MMT}, see also \cite{KM} for a generalization.
For the convenience of the reader we sketch its proof based on \cite[Theorem 23]{MMT}. 
\begin{theorem}\label{thm:minimizer} 
Let $u_0\in \W$ be a local minimum point of $I$ in $C^1_0(\overline{\Omega})$-topology, namely, there exists $\varepsilon>0$ such that
\begin{equation}\label{eq:prop-4-4}
I(u_0)\le I(u_0+h) 
\quad \text{for any}~ 
h\in C^1_0(\overline{\Omega})
~\text{satisfying}~ 
\|h\|_{C^1_0(\overline{\Omega})} < \varepsilon.
\end{equation}
Then $u_0$ is also a local minimum point of $I$ in $\W$-topology. 
\end{theorem} 
\begin{proof}
Since $\langle I^\prime(u_0), h\rangle =0$ for every 
$h \in C^1_0(\overline{\Omega})$ and since $C^1_0(\overline{\Omega})$ is dense in $\W$, we deduce that $u_0$ is a critical point of $I$. That is, 
\begin{equation}\label{prop-4-1}
-\Delta_p u_0 - \Delta_q u_0 
=f(x,u_0) \quad \text{in}~ \Omega, 
\end{equation}
in the weak sense.
Moreover, one can show that $u_0 \in  C_0^{1,\nu}(\overline{\Omega})$ for some $\nu \in (0,1)$, cf.\ Remark \ref{rem:positive} or \cite[Section 2.4]{marcomasconi}. 

Suppose, by contradiction, that $u_0$ is not a local minimum point of $I$ in $\W$-topology. 
Then for any sufficiently small $\varepsilon>0$ we have
\begin{equation*}\label{prop-4-2} 
m_\varepsilon
:=
\inf
\left\{
I(u_0+h):~ h \in \widetilde{B}_\varepsilon(0)
\right\} 
<I(u_0), 
\end{equation*}
where $\widetilde{B}_\varepsilon(0) = \{v\in\W: \|v\|_r \le \varepsilon\}$ and $r\in [1,p^*)$ is given in the assumption \ref{G}. 
Let $\{h_n\}$ be a minimizing sequence for $m_\varepsilon$ with a fixed $\varepsilon>0$. Thanks to \ref{G}, $\{h_n\}$ is bounded in $\W$.
Therefore, $m_\varepsilon$ is attained by some $h_\varepsilon \in \widetilde{B}_\varepsilon(0)$, since 
$I$ is weakly lower semicontinuous on $\W$ and $\widetilde{B}_\varepsilon(0)$ is weakly closed in $\W$. 
Then, due to the Lagrange multipliers rule, there exists 
$\lambda_\varepsilon \le 0$ such that 
\begin{equation}\label{prop-4-3} 
-\Delta_p (u_0+h_\varepsilon) -\Delta_q (u_0+h_\varepsilon) =f(x,u_0+h_\varepsilon)
+\lambda_\varepsilon\,|h_\varepsilon|^{r-2}h_\varepsilon \quad 
{\rm in}\ \Omega. 
\end{equation}
Denoting now
\begin{align*}
\widetilde{A}(x,y) 
=
|\nabla u_0(x)+y|^{p-2}(\nabla u_0(x)+y)
&+
|\nabla u_0(x)+y|^{q-2}(\nabla u_0(x)+y)
\\ 
&-|\nabla u_0(x)|^{p-2}\nabla u_0(x)
-
|\nabla u_0(x)|^{q-2}\nabla u_0(x),
\end{align*}
we subtract \eqref{prop-4-1} from \eqref{prop-4-3} and get
\begin{equation*}\label{prop-4-5} 
-{\rm div}\, \widetilde{A}(x,\nabla h_\varepsilon)=f(x,u_0+h_\varepsilon)
-f(x,u_0) +\lambda_\varepsilon\,|h_\varepsilon|^{r-2}h_\varepsilon
\quad \text{in}~ \Omega.
\end{equation*} 
Recalling that $\lambda_\varepsilon \leq 0$ and using the Moser iteration method 
(see, e.g., \cite[Theorem C]{MMT}),  
we can find $M_1>0$ independent of $\varepsilon$ such that 
$\|h_\varepsilon\|_\infty \le M_1$ for every $\varepsilon>0$. 
Then, applying the regularity result of \cite{L} to the solution $u_0+h_\varepsilon$ of \eqref{prop-4-3}, we obtain that  
$u_0+h_\varepsilon\in C_0^1(\overline{\Omega})$, and so 
$h_\varepsilon\in C_0^1(\overline{\Omega})$ for every $\varepsilon>0$. 

Finally, it can be shown as in \cite[Theorem 23]{MMT} that there exists $d_0>0$ such that
$$
|\lambda_\varepsilon |h_\varepsilon(x)|^{r-2}h_\varepsilon(x)| \le d_0
\quad \text{for every}~ x \in \Omega
~\text{and}~ \varepsilon>0.
$$ 
This implies that $f(x,u)
+\lambda_\varepsilon|h_\varepsilon(x)|^{r-2}h_\varepsilon(x)$ is bounded on $\Omega\times[-M_1-\|u_0\|_\infty,M_1+\|u_0\|_\infty]$ 
uniformly in $\varepsilon>0$. 
Thus, applying again the regularity result of \cite{L} to the solution 
$u_0+h_\varepsilon$ of \eqref{prop-4-3}, we deduce the existence of $\theta\in(0,1)$ and $M_2>0$, both independent of $\varepsilon$, such that $u_0+h_\varepsilon \in C_0^{1,\theta}(\overline{\Omega})$ and 
$\|u_0+h_\varepsilon\|_{C_0^{1,\theta}(\overline{\Omega})} \le M_2$ 
for every $\varepsilon>0$. Since $C_0^{1,\theta}(\overline{\Omega})$ 
is embedded compactly into $C_0^1(\overline{\Omega})$, we infer that 
$u_0+h_\varepsilon \to u_0$ as $\varepsilon \searrow 0$ 
in $C_0^1(\overline{\Omega})$ 
by noting that $h_\varepsilon\to 0$ in $L^r(\Omega)$ 
as $\varepsilon \searrow 0$. 
Consequently, we get the following contradiction between \eqref{eq:prop-4-4} and \eqref{prop-4-1}:
\[
I(u_0+h_\varepsilon)
=
m_\varepsilon <I(u_0) 
\le 
I(u_0+h_\varepsilon)
\quad \text{for all sufficiently small}~ \varepsilon>0.
\qedhere
\]
\end{proof}

\section*{Acknowledgements}
V.~Bobkov was supported in the framework of executing the development program of Scientific Educational Mathematical Center of Privolzhsky Federal Area, additional agreement no.~075-02-2020-1421/1 to agreement no.~075-02-2020-1421.
M.~Tanaka was supported by JSPS KAKENHI Grant Number JP 19K03591.


\end{document}